\documentclass[11pt,leqno,twoside,reqno]{amsart}

\usepackage{amssymb,amsmath,amsthm,soul,color,stmaryrd,enumitem} 
\usepackage{t1enc}
\usepackage[cp1250]{inputenc}
\usepackage{a4,indentfirst,latexsym}
\usepackage{graphics}
\usepackage{mathrsfs}
\usepackage{cite,graphicx}
\usepackage[colorlinks=true,urlcolor=blue,citecolor=red,linkcolor=blue,linktocpage,pdfpagelabels,bookmarksnumbered,bookmarksopen]{hyperref}
\usepackage[english]{babel}
\usepackage[left=2.61cm,right=2.61cm,top=2.72cm,bottom=2.72cm]{geometry}

\usepackage[metapost]{mfpic}
\usepackage[hyperpageref]{backref}
\usepackage[colorinlistoftodos]{todonotes}

\newcommand{\R}{\mathbb{R}}

\newcommand{\RN}{{\mathbb{R}^N}}
\newcommand{\RD}{{\mathbb{R}^2}}

\renewcommand{\le}{\leslant}
\renewcommand{\ge}{\geslant}

\renewcommand{\l }{\lambda}
\newcommand{\n }{\nabla }

\renewcommand{\t}{\theta}

\renewcommand{\L}{\Lambda}

\newcommand{\Hr}{H^1_r(\RD)}

\newcommand{\N}{\mathbb{N}}

\renewcommand{\o}{\omega}

\newcommand{\irn }{\int_{\RN}}

\newcommand{\ird}{\int_{\RD}}

\def\bbm[#1]{\mbox{\boldmath $#1$}}
\newcommand{\beq }{\begin{equation}}
\newcommand{\eeq }{\end{equation}}

\renewcommand{\le}{\leqslant}
\renewcommand{\ge}{\geqslant}
\newcommand{\dis}{\displaystyle}

\makeatletter
\providecommand\@dotsep{5}
\def\listtodoname{List of Todos}
\def\listoftodos{\@starttoc{tdo}\listtodoname}
\makeatother

\numberwithin{equation}{section}

\newtheorem{theorem}{Theorem}[section]
\newtheorem{lemma}[theorem]{Lemma}
\newtheorem{proposition}[theorem]{Proposition}
\newtheorem{cor}[theorem]{Corollary}
\newtheorem{remark}[theorem]{Remark}

\title[Chern-Simons-Schr\"odinger equation]{A multiplicity result for Chern-Simons-Schr\"odinger equation with a general nonlinearity}

\author[P. L. Cunha]{Patricia L. Cunha}
\author[P. d'Avenia]{Pietro d'Avenia}
\author[A. Pomponio]{Alessio Pomponio}
\author[G. Siciliano]{Gaetano Siciliano}

\address[P. L. Cunha]{\newline\indent
Departamento de Inform\'atica e M\'etodos Quantitativos
\newline\indent
Funda\c c\~ao Getulio Vargas
\newline\indent
Av. Nove de Julho, 2029 - Bela Vista
\newline\indent
01313-902, S\~ao Paulo - SP, Brazil.}
\email{\href{mailto:patcunha80@gmail.com}{patcunha80@gmail.com}}

\address[P. d'Avenia]{\newline\indent
Dipartimento di Meccanica, Matematica e Management 
\newline\indent
Politecnico di Bari
\newline\indent
Via E. Orabona 4
\newline\indent
70125 Bari, Italy.}
\email{\href{mailto:pietro.davenia@poliba.it}{pietro.davenia@poliba.it}}

\address[A. Pomponio]{\newline\indent
Dipartimento di Meccanica, Matematica e Management 
\newline\indent
Politecnico di Bari
\newline\indent
Via E. Orabona 4
\newline\indent
70125 Bari, Italy.}
\email{\href{mailto:alessio.pomponio@poliba.it}{alessio.pomponio@poliba.it}}

\address[G. Siciliano]{\newline\indent
Departamento de Matem\'atica
\newline\indent
Universidade de S\~ao Paulo
\newline\indent
Rua do Mat\~ao, 1010 - Cidade Universitaria
\newline\indent
05508-090, S\~ao Paulo - SP, Brazil.}
\email{\href{mailto:sicilian@ime.usp.br}{sicilian@ime.usp.br}}

\thanks{P. d'Avenia and A. Pomponio are supported by GNAMPA project  ``\emph{Aspetti differenziali e 
geometrici nello studio di problemi  ellittici quasilineari}''.
G. Siciliano is supported by Fapesp (SP) and CNPq, Brazil.}
\date{\today}

\begin{document}

\pretolerance10000 

\begin{abstract}
In this paper we give a multiplicity result for the following Chern-Simons-Schr\"{o}dinger equation
\begin{equation*}
-\Delta u+2q u \int_{|x|}^{\infty}\frac{u^{2}(s)}{s}h_u(s)ds  +q u\frac{h^{2}_u(|x|)}{|x|^{2}} = g(u), \quad\hbox{in }\RD,
\end{equation*}
where $\displaystyle h_u(s)=\int_0^s \tau u^2(\tau) \ d \tau$, under very general assumptions on the nonlinearity $g$. 
In particular, for every $n\in \mathbb N$, we prove the existence of (at least) $n$ distinct solutions, 
for every $q\in (0,q_{n})$, for a suitable $q_n$. 
\end{abstract}

\maketitle

\section{Introduction}

This paper is devoted to the study of the following Chern-Simons-Schr\"{o}dinger equation
\begin{equation}\label{corretta}
-\Delta u+2q u \int_{|x|}^{\infty}\frac{u^{2}(s)}{s}h_u(s)ds  +q u\frac{h^{2}_u(|x|)}{|x|^{2}} = g(u) \quad\hbox{in }\RD,
\end{equation}
where $u:\mathbb R^{2}\to \mathbb R$, $q>0$
and $ h_u(s)=\int_0^s \tau u^2(\tau) \ d \tau$. This equation describes nonrelativistic matter interacting with Chern-Simons gauge 
fields in the plane, under a suitable ansatz.
For the reader convenience, in Section \ref{se:deduction}, we will give the physical motivations of this model, the derivation and 
justification of \eqref{corretta}.  

Arguing as in \cite[Proposition 2.2]{BHS}, solutions of \eqref{corretta} are critical points of the $C^1$-functional 
\begin{equation}\label{J}
J_q(u)=\frac{1}{2}\ird |\nabla u|^{2} \ dx+\frac{q}{2}\int_{\mathbb R^{2}}
\frac{u^{2}(|x|)}{|x|^{2}} \Big(\int_{0}^{|x|}s u^{2}(s)ds\Big)^{2}dx- \int_{\mathbb R^{2}}G(u)dx
\end{equation}
defined on $H_r^1(\RD)=\{u\in H^1(\RD): u \hbox{ is radially symmetric}\}$, where $G(s)=\int_{0}^{s}g(\tau)d\tau$. 
As we can see, the functional includes a nonlocal term which requires a careful analysis.

Our aim is to study \eqref{corretta} with a Berestycki, Gallou\"et \& Kavian type nonlinearity \cite{BGK}, 
which satisfies very general and almost necessary conditions, and also it is the planar version of the Berestyski-Lions type 
nonlinearity \cite{BL1,BL2}. These papers concern with the existence of nontrivial radial  solutions for the following autonomous nonlinear 
field equation
\beq \label{eq:BL}	
\left\{\begin{array}{cl}
-\Delta u= g(u) & \hbox{in }\RN,
\\
u\in H^1(\RN),
\end{array}
\right.
\eeq
with $N\ge 2$, by assuming 
\begin{enumerate}[label=(g\arabic*),ref=g\arabic*,start=0]
\item \label{vecchiaitg0}
$g\in C(\mathbb R, \mathbb R)$ is an odd function.
\item \label{vecchiaitg1}
For $N\ge 3$, 
$$ \limsup_{\xi\to\infty}\frac{g(\xi)}{\xi^{\frac{N+2}{N-2}}}\le0.$$ 
For $N=2$, 
$$ \limsup_{\xi\to\infty}\frac{g(\xi)}{e^{\alpha\xi^{2}}}\le 0,\ \ \forall\, \alpha>0.$$  
\item \label{vecchiaitg2}
For $N\ge 3$, 
\[
\displaystyle-\infty<\liminf_{\xi\to 0}\frac{g(\xi)}{\xi} \le \limsup_{\xi\to 0}\frac{g(\xi)}{\xi}<0.
\]
For $N=2$,
\[
\displaystyle-\infty<\lim_{\xi\to 0}\frac{g(\xi)}{\xi} <0.
\]
\item \label{vecchiaitg3}
There exists $\zeta_{0}>0$ such that $G(\zeta_{0})>0$.
\end{enumerate}

As we can see, there is a difference in the assumption \eqref{vecchiaitg2}  between the cases $N\ge 3$ and $N=2$. 
The existence of a limit $\lim_{\xi \to 0} g(\xi)/\xi \in (-\infty, 0)$, when $N=2$, is essential in 
the arguments of Berestycki, Gallou\"et \& Kavian \cite{BGK} to prove the Palais-Smale ((PS) from now on) condition for the corresponding functional under suitable constraint.

Later on, in a more recent paper, Hirata, Ikoma \& Tanaka \cite{HIT} consider equation \eqref{eq:BL} in the case $N\ge 2$ 
assuming \eqref{vecchiaitg0}, \eqref{vecchiaitg1}, \eqref{vecchiaitg3} and 

\begin{enumerate}[label=(g\arabic*'),ref=g\arabic*',start=2]
\item \label{vecchiaitg2'}
$\displaystyle-\infty<\liminf_{\xi\to 0}\frac{g(\xi)}{\xi} \le \limsup_{\xi\to 0}\frac{g(\xi)}{\xi}<0$
\end{enumerate}
and find radial solutions as critical points of the functional 
\[ 
I(u)=\frac{1}{2}\irn |\n u |^2\ dx - \irn G(u)\ dx.
\]
Following the approach introduced by Jeanjean in \cite{J}, Hirata, Ikoma \& Tanaka \cite{HIT} consider 
the auxiliary functional 
$\tilde{I}:\R\times H_r^1(\RN) \to \R$
\[
\tilde{I}(\t,u)=\frac{e^{(N-2)\t}}{2}\irn |\n u |^2\ dx - e^{N\t} \irn G(u)\ dx.
\]
In this way they are able to find a (PS) sequence $(\t_j,u_j)$ such that $\t_j\to 0$ and $u_j$ ``almost'' satisfies the
Poho\v{z}aev identity associated to \eqref{eq:BL}. Using this extra information, it is proved that problem \eqref{eq:BL} 
possesses a positive least energy solution and infinitely many (possibly sign changing) radially symmetric solutions.

%
We prove
\begin{theorem}\label{Main}
	Assume \eqref{vecchiaitg0}, \eqref{vecchiaitg1}, \eqref{vecchiaitg2'} and \eqref{vecchiaitg3}. For every $n\in \mathbb N$ there exists $q_{n}>0$ 
	such that, for every $q\in (0,q_{n})$, equation \eqref{corretta} admits (at least) $n$ distinct solutions.
\end{theorem}

Of course the solutions will appear in pairs, since the functional is even. 
Moreover as in \cite[Proposition 2.2]{BHS} it can be seen that the solutions are classical.

Due to the presence of the nonlocal term, the method in \cite{HIT} seems to be not sufficient and so we have to combine it with a penalization 
technique introduced in \cite{BB,JC}. A similar approach has been used in \cite{ADP2} to the study a perturbed version of \eqref{eq:BL}, in 
the case $N\ge 3$, introducing a family of auxiliary functionals suitably related to the ``original'' one. In our case, we consider the same 
family of functionals but we cannot repeat the same arguments.

One of the main difficulties and one of the main differences regarding the previous works is the proof of a uniform boundedness in $\Hr$ 
of a suitable (PS) sequence $(u_j)_j$ with respect to $q$, at least for small positive $q$. In \cite{ADP2} the key point is to show a 
suitable bound of $(u_j)_j$ in $D^{1,2}(\RN)$, and then the boundedness of the same sequence in $L^2(\RN)$ follows by the continuous embedding 
of $D^{1,2}(\RN)$ into $L^{2^*}(\RN)$, where $N\ge 3$ and $2^*=2N/(N-2)$. In two dimensional case,  however, due to the lack of embedding theorems 
of $D^{1,2}(\RD)$ into $L^p$-spaces, these arguments do not work. We get a such uniform boundedness using an appropriate diagonalization argument.
Moreover, while in \cite{ADP2},  the hypothesis  \eqref{vecchiaitg2} guarantees suitable compactness properties, in our case the assumption 
\eqref{vecchiaitg2'} creates several problems with the lack of compactness. This is also the reason why our problem does not fall even into the case 
considered in \cite{JS}.

We recall now some recent results concerning the equation under study.
The case  $g(u)=|u|^{p-1}u-\o u$, where $p>1$ and $\o>0$ is the phase of the standing wave $\psi(t,x)=u(x)e^{i\omega t}$,
has been considered by several authors showing how the existence of solutions and the geometrical aspect of the Euler-Lagrange 
functional associated to \eqref{corretta} depend strongly on $p$ and  $\o$.
In \cite{BHS}, Byeon, Huh \& Seok  show that, for $p\geq3$,  the Euler-Lagrange functional  is unbounded from below and exhibits a mountain-pass geometry for $p>3$.  However the existence of a  solution is not so direct, since for $p \in (3,5)$ the (PS) condition is not known to hold. This problem is bypassed  by using a constrained minimization taking into account the Nehari and Poho\v{z}aev identities. In the special case $p=3$, with a suitable choice of physical constants, they prove that solutions can be found by passing to a self-dual equation, which leads to a Liouville equation that can be solved explicitly. Those are the unique positive solutions. Finally, in the same paper, for $p\in(1,3)$, solutions are found as minimizers on a $L^2$-sphere: the value $\o$ comes out as a Lagrange multiplier and it is not controlled. \\
Later, the result for $p \in  (1,3)$ has been extended by Pomponio \& Ruiz \cite{PR1} 
by investigating the geometry of the Euler-Lagrange  functional. 
Through a careful analysis for a limit equation, they showed that there exist $0 < \o_0 < \tilde \o < \bar \o$ such that if $\o> \bar \o$, 
the unique solutions to \eqref{corretta} are the trivial ones; if $\tilde \o < \o < \bar \o$, there are at least two positive solutions to 
\eqref{corretta}; if $0 < \o < \tilde \o$, there is a positive solution to \eqref{corretta} for almost every $\o$. \\
Moreover Pomponio \& Ruiz also study in \cite{PR2} the case of a bounded domain for $p \in  (1, 3)$ and obtain some results on boundary 
concentration of solutions.
\\
Finally we mention the paper of Wan \& Tan \cite{WT} where they consider $g(u)=f(u)-\o u$ with $f$ asymptotically linear
and the paper of  Huh \cite{huh2} where
infinitely many solutions (possibly sign-changing) have been found  for $p>5$.

\begin{remark}
The fact that we obtain  nontrivial solutions only for sufficiently small $q$ is not surprising and we can not 
expect more. Indeed, in \cite{PR1} the  equation 
\begin{equation}\label{eq:pr}
-\Delta u+ \omega u+ 2u \int_{|x|}^{\infty}\frac{u^{2}(s)}{s}h_u(s)ds  + u\frac{h^{2}_u(|x|)}{|x|^{2}} = |u|^{p-1}u, \qquad \hbox{in }\RD
\end{equation}
is considered.
Performing the rescaling $u\mapsto u_\o=\o^\frac{1}{p-1}\ u(\o^\frac 12 \ \cdot)$ and defining $q=\o^\frac{2(3-p)}{p-1}$, equation \eqref{eq:pr} 
becomes  \eqref{corretta}, with $g(u)=-u+|u|^{p-1}u$.
By \cite[Theorem 1.2]{PR1}, for $p\in (1,3)$,  there is a solution of \eqref{corretta} with $g(u)=-u+|u|^{p-1}u$ only for sufficiently small $q$.
Hence, by means of Theorem \ref{Main}, we have now a more precise picture on the existence of solutions of \eqref{eq:pr}. There exist, indeed, 
$\o_0<\tilde{\o}<\bar \o$  and a decreasing sequence $(\o_n)_{n\ge 1}$ such that the following happens
\vspace{.1cm}
\begin{center}
\begin{tabular}{|l|l|l|l|}
\hline
for all $\o\in (\o_{n+1},\o_n) $
&
for a.e. $\o\in (\o_1,\o_0)$
&
for all $\o\in (\o_0,\tilde{\o})$
&
for all $\o>\bar\o$   
\\
\hline
at least $n$ (possibly sign-changing)
&
at least a positive
&
at least two positive
&
no nontrivial  
\\
solutions &  solution &  solutions &  solutions \\
\hline
\end{tabular}
\end{center}
\end{remark}

The paper is organized as follows. In Section \ref{se:deduction}, for the sake of completeness, we derive the equation \eqref{corretta}. 
It is obtained by the complete set of the Euler-Lagrange equations  of the Lagrangian density given e.g. in  \cite[eqn. (8)]{JP1} and by 
considering a more general nonlinearity, under a suitable ansatz. Then Section \ref{se:main} is devoted to prove Theorem \ref{Main}.

Throughout the paper, we denote by $C$  generic positive constants  specifying, if necessary, the parameters on 
which they depend. The constants may also change from line to line.

\section{Deduction of the equation}\label{se:deduction}

Let us consider the  three dimensional Lorentz space time $\mathbb R^{1,2}$ with metric tensor $\operatorname{diag}(+1,-1,-1)$
and coordinates $x^{\mu}=(ct,x^{1},x^{2})$, where $c$ is the  speed of light.
As usual, we adopt the Einstein convention on repeated indices, where greek indices always vary in $\{0,1,2\}$ while latin ones run in $\{1,2\}$.
In particular we will always distinguish in the following
between covariant and contravariant indices, which are obtained
ones from the others by using the metric.

The starting point for our equation is the Schr\"odinger Lagrangian density of the matter field
\begin{equation}
\label{LS}
\mathcal{L}_{\rm S}(\psi)= i \hbar \overline \psi  \partial_t\psi -\frac{\hbar^{2}}{2m} |\nabla \psi|^{2}+2W(\psi)
\end{equation}
where $\psi:\mathbb{R}^{1,2}\to\mathbb{C}$, $m>0$ is the mass parameter, $\hbar=h/2\pi$ is the normalized Planck constant and 
$W:\mathbb C\to \mathbb R$ is a  nonlinearity of type $W(\psi)=W(|\psi|)$ which describes the interaction among many particles. 

Let us define the electromagnetic tensor as
\[
F_{\mu\nu}=\partial_{\mu} A_{\nu}-\partial_{\nu}A_{\mu}
\]
where $A^{\mu}=(A^{0},\mathbf A)=(A^0,A^1,A^2)$ is the gauge potential.
We observe explicitly that, in virtue of the choice of the metric,
\begin{equation}\label{eq:cov-controv}
(A^{0}, A^{1},A^{2})=(A_{0},-A_{1},-A_{2})
\end{equation}
and indeed we will use, case by case, the notation which is more convenient.
In vectorial notation the electromagnetic field is given by
\begin{equation}
\label{EB}
\mathbf E=-\nabla A^{0}-\frac{1}{c}\partial_{t}\mathbf A, \qquad B= \nabla \times  \mathbf A.
\end{equation}

To study the interaction between the matter (expressed by the wave function $\psi$) and the electromagnetic field $(\mathbf{E},B)$ given by \eqref{EB}, 
we consider the gauge (or Weyl) covariant derivatives
$$ 
D_{t}= \partial_{t}+\frac{ie}{\hbar} A^{0},
\qquad
\mathbf D=\nabla -\frac{ie}{\hbar c}  \mathbf A,
$$
where $e$ is a coupling constant. 
We point out that these operators are obtained by the so called {\sl minimal coupling rule}, see e.g. \cite{Fel,Naber}.
Then we substitute in \eqref{LS} the derivatives with the covariant ones, getting
\[
\tilde{\mathcal{L}}_{\rm S}(\psi,A^0,\mathbf{A})= i \hbar \overline \psi  D_{t}\psi -\frac{\hbar^{2}}{2m} |\mathbf D \psi|^{2}+2W(\psi).
\]
However to obtain  the complete Lagrangian we need to consider also the term involving  the gauge potentials,
since they are unknown. In $\mathbb R^{1,2}$ one can take 
the more general term
\begin{equation}
\label{LDCS}
\mathcal{L}_{\rm MCS}=\underbrace{-\frac{1}{4} F^{\mu\nu} F_{\mu\nu}}_{\mathcal L_{\rm Max}}+ \underbrace{\frac{\kappa}{4}\varepsilon^{\mu\alpha\beta}A_{\mu}F_{\alpha\beta}}_{\mathcal{L}_{\rm CS}}
\end{equation}
which involves not only the usual Maxwell term but also the so-called Chern-Simons term.
Here $\varepsilon$ is the Levi-Civita tensor and $\kappa\in \mathbb R$ is a parameter which controls the Chern-Simons term.
Strictly speaking \eqref{LDCS} is the Lagrangian of the gauge potentials in the vacuum.

Thus the total Lagrangian density is
\begin{align*}
\mathcal L_{\rm tot}(\psi,A^0,\mathbf{A})
:= & 
\mathcal{L}_{\rm MCS}(A^0,\mathbf{A}) + \tilde{\mathcal{L}}_{\rm S}(\psi,A^0,\mathbf{A})\\
= & 
-\frac{1}{4} F^{\mu\nu} F_{\mu\nu}+ \frac{\kappa}{4}\varepsilon^{\mu\alpha\beta}A_{\mu}F_{\alpha\beta}
+i \hbar \overline \psi  D_{t}\psi -\frac{\hbar^{2}}{2m} |\mathbf D \psi|^{2}+2W(\psi).
\end{align*}

Due to the presence of the Chern-Simon term, the total Lagrangian is not  invariant 
with respect to the gauge transformations
\begin{equation}\label{invariance}
A^{\mu}\rightarrow A^{\mu}+\partial^{\mu}\chi,
\qquad
\psi \rightarrow \psi e^{\frac{ie}{\hbar c} \chi},
\qquad
\chi\in C^{\infty}(\mathbb R^{1,2}),
\end{equation}
nevertheless, its  Euler-Lagrange equations are gauge invariant.

As discussed in \cite{JP1}, the Chern-Simons electrodynamics
is obtained by taking the formal limit $|\kappa|\to \infty$ in the topologically massive model.
Indeed at large distances and low energies the lower derivatives
of the Chern-Simons term dominate the higher derivative appearing in the Maxwell term; hence this last 
term becomes negligible and the above total Lagrangian reduces to
\begin{equation*} 
\mathcal L (\psi,A^0,\mathbf{A})
=  \frac{\kappa}{4}\varepsilon^{\mu\alpha\beta}A_{\mu}F_{\alpha\beta}
+i \hbar \overline \psi  D_{t}\psi -\frac{\hbar^{2}}{2m} |\mathbf D \psi|^{2}+2W(\psi).
\end{equation*}
For this and further discussions see also  \cite{DJT,Ha,JP2,Sc,Tar}. This is the Lagrangian we are interested in.
By taking the variations of the action functional 
$\mathcal S=\iint \mathcal L dx dt$ with respect to $\psi,A_{\mu}$, recalling \eqref{eq:cov-controv}, we have the following Euler-Lagrange equations
\begin{equation}\label{EL}
\begin{split}
& i\hbar D_{t} \psi+\frac{\hbar^{2}}{2m} \mathbf D^{2}\psi =- W'(\psi)\\ 
& \kappa(\partial_{1}A_{2}-\partial_{2}A_{1})=e |\psi|^{2}\\ 
&\kappa(\partial_{2}A_{0}-\partial_{0}A_{2})= \frac{e}{c}\frac{\hbar}{m}\mathfrak{Im} ( \overline\psi D_{1} \psi)\\ 
&\kappa(\partial_{0}A_{1}-\partial_{1}A_{0})= \frac{e}{c}\frac{\hbar}{m}\mathfrak{Im} ( \overline\psi D_{2} \psi). 
\end{split}
\end{equation}
Note that these equations are invariant under the gauge transformations \eqref{invariance}.

If we define
\begin{equation*}\label{correnti}
 J^{\mu}=(c\rho,{\mathbf{J}}):= \Big(c \overline \psi \psi, \frac{\hbar}{2mi}(\overline\psi \mathbf D \psi -\psi \overline{\mathbf D \psi})\Big)
 =\Big(c |\psi|^{2} , \frac{\hbar}{m} \mathfrak{Im}(\overline \psi {\bf D}\psi)\Big),
\end{equation*}
 the last three equations in \eqref{EL} can be written,
 since $\mathfrak{Im}(\overline \psi  D^j \psi)=\mathfrak{Im}(\overline \psi  D_j \psi)$, as
\begin{equation}\label{CSeq}
\frac{\kappa}{2}\varepsilon^{\mu\alpha\beta}F_{\alpha\beta}=\frac{e}{c}J^{\mu}
\end{equation}
which are the gauge field equations of the Chern-Simons electrodynamics. Thus,
$J^{\mu}$ can be interpreted as {\sl a current density}.
In particular, we have the continuity equation for the currents $$\partial_{\mu}J^{\mu}=\partial_t \rho+\nabla\cdot\mathbf{J}=0.$$
Moreover, in terms of the electromagnetic field $(\mathbf E, B)$, equations \eqref{CSeq} are
\[
\kappa c B=-e J^{0},
\quad
\kappa c E^1=eJ^2,
\quad
\kappa c E^2=-eJ^1.
\]
The first equation yields the remarkable fact that
if $B=0$, then $\psi=0$ and that
any field configuration with total charge $Q(t)=e\int |\psi(t,x)|^{2}dx$ also carries a magnetic flux $\Phi(t)=\int B(t,x)dx$ given by
$$\Phi(t)=-\frac{1}{\kappa}Q(t),$$ and indeed they are conserved quantities (i.e. constant in time)
for ``well behaved'' $\psi$. This provide an explicit realization of anyons, see \cite{L,W}.
On the other hand, the other  two equations give the interesting identities
$$\nabla \cdot \mathbf E= \frac{e}{c\kappa}\nabla \times \mathbf J
\ \ \text{ and } \ \ \nabla\cdot \mathbf J=-\frac{c\kappa}{e}\nabla \times\mathbf E,$$
so the matter and the electromagnetic field support each other.

However, it is convenient to  write $\psi(t,x)=u(t,x)e^{i S(t,x)}$; 
hence by taking the variations of the action with respect to $u,S, A_{\mu}$
we get the following set of equations
\begin{equation}\label{ELS}
\begin{split}
&-\frac{\hbar^{2}}{2m}\Delta u+ \Big(\hbar\partial_{t}S+e A^{0}+\frac{\hbar^{2}}{2m}|\nabla S|^{2}
+\frac{e^{2}}{2mc^{2}}|\mathbf A|^{2}-\frac{\hbar e}{mc} \nabla S \cdot \mathbf A\Big)u= W' (u)\\ 
&\partial_{t}u^{2}+\frac{\hbar}{m}\nabla \cdot \left[\left(\nabla S - \frac{e}{\hbar c}\mathbf{A}\right)u^2\right]=0\\ 
& 
\kappa(\partial_{1}A_{2}-\partial_{2}A_{1})=e u^{2} \\ 
& 
\kappa(\partial_{2}A_{0}-\partial_{0}A_{2})=\frac{e\hbar}{cm}\left(\partial_1 S + \frac{e}{\hbar c}A_1\right)u^2\\ 
& 
\kappa(\partial_{0}A_{1}-\partial_{1}A_{0})=\frac{e\hbar}{cm}\left(\partial_2 S + \frac{e}{\hbar c}A_2\right)u^2. 
\end{split}
\end{equation}
The second equation is of course a continuity equation  (conservation law of currents),
as $\mathcal L$ does not depend explicitly
on $S$, but just on its  derivatives.


In the static case $A^{\mu}=A^{\mu}(x)$, by using the Helmholtz decomposition and  taking a suitable gauge, without loss of generality we can assume 
that $\lim_{|x|\to+\infty}A^0(x)=0$ (if we suppose that such a limit exists) and that $\partial_{j}A^{j}=0$ (Coulomb gauge).
%
%
An interesting case appears when we take the ansatz of standing waves
solutions  $\psi(t,x)=u(x)e^{i\omega t}$. 
Then the previous equations \eqref{ELS} become
\begin{equation}\label{ELF}
\begin{split}
&-\frac{\hbar^{2}}{2m}\Delta u+ \Big(\hbar\omega+e A^{0}+\frac{e^{2}}{2mc^{2}}|\mathbf A|^{2}\Big)u=W'(u)\\
&{\bf A}\cdot \nabla u^{2}=0\\
& \kappa(\partial_{1}A_{2}-\partial_{2}A_{1})=e u^{2} \\
& \kappa\partial_{2}A_{0}=\frac{e^2}{mc^2}A_1u^2\\ 
& \kappa\partial_{1}A_{0}=-\frac{e^2}{mc^2}A_2u^2.
\end{split}
\end{equation}

Arguing as in \cite{huh2}, the Coulomb gauge and the third equation in \eqref{ELF}
imply that $A_{1}, A_{2}$ are uniquely determined by $u,$
since they satisfy 
$$\Delta A_{1}=-\frac{e}{\kappa}\partial_{2} u^{2}, 
\quad
\Delta A_{2}=\frac{e}{\kappa}\partial_{1}u^{2}.$$
Hence they are given by
$$A_{1}=-\frac{e}{\kappa} G_{2}* u^{2}, 
\quad
A_{2}=\frac{e}{\kappa}G_{1}*u^{2},
\ \  \hbox{where } \ \
G_{i}(x)=\frac{1}{2\pi}\frac{x^{i}}{|x|^{2}}.$$
Coming back to the last two equations in \eqref{ELF} we infer that
$$-\Delta A_{0}= \frac{e^{2}}{\kappa m c^{2}} \left[\partial_{1}(A_{2}u^{2}) - \partial_{2}(A_{1}u^{2}) \right]$$
and hence
$$A_{0}=\frac{e^{2}}{\kappa m c^{2}}\left[ G_{2}* (A_{1}u^{2}) - G_{1}* (A_{2}u^{2}) \right].$$
Observe also that from the second equation in \eqref{ELF} it follows that, up to the ``trivial cases'',
the function $u$ is radial  if and only if $\mathbf A$ is a {\em tangential} vector field, i.e. of type
$$\mathbf A=\frac{e}{\kappa}h_{u}(x)\mathbf t,
\quad
\hbox{where }
\mathbf t=(x^{2}/|x|^{2},-x^{1}/|x|^{2}).$$ 
Moreover, since the problem is invariant by translations, to avoid the related difficulties, we look for radial solutions $u$. 
Thus, from this choice, arguing as in \cite[Lemma 3.3]{huh2}, it follows that
$\mathbf A$ has to be invariant for the group action
$${\rm T}_{g}\mathbf A(x)= g^{-1}\cdot \mathbf A (g(x)), 
\quad
g\in O(2),$$
and this readily implies that $h_{u}$ has to be a radial function.
Summing up, whenever $u$ is radial, the magnetic potential has to be necessarily written as
\begin{equation}\label{A}
A^{1}(x)=\frac{e}{\kappa}\frac{x^{2}}{|x|^{2}}h_{u}(|x|),
\quad
A^{2}(x)=-\frac{e}{\kappa}\frac{x^{1}}{|x|^{2}}  h_{u}(|x|).
\end{equation}
Finally, by \eqref{ELF} we see that 
$$\nabla A^{0}=\frac{e^{3}}{mc^{2}\kappa^{2} } u^{2}(|x|)h_{u}(|x|)  \mathbf n,
\ \ \text{where } \ \mathbf n=(x^{1}/|x|^{2}, x^{2}/|x|^{2})\,;
$$
in other words, the electric potential is radial and so can be written as
\begin{equation}\label{A0}
A^{0}(x)=A^{0}(|x|).
\end{equation}

Now we can find the explicit dependence of $A^{0}$ and $h_{u}$ from the function $u$;
indeed, by substituting \eqref{A} and \eqref{A0} into \eqref{ELF} 
one find (assuming $h_{u}(0)=0$, which is necessary to have $\mathbf A$ smooth)
\[
h_{u}(|x|)=
\int_{0}^{|x|}\tau u^{2}(\tau)d\tau,
\quad
A^{0}(|x|)=\frac{e^{3}}{ m c^2\kappa^{2}} \int_{|x|}^{\infty}\frac{u^{2}(\tau)}{\tau} h_{u}(\tau) d\tau.
\] 
With these expressions in hands the first equation in \eqref{ELF}, the  equation of the matter field, 
is
\begin{equation*} 
-\frac{\hbar^{2}}{2m}\Delta u+\hbar\omega u+
\frac{e^{4}}{mc^{2}\kappa^{2}}u\int_{|x|}^{\infty}\frac{u^{2}(s)}{s}h_{u}(s)ds+
\frac{e^{4}}{2mc^{2}\kappa^{2}}u\frac{h_{u}^{2}(|x|)}{|x|^{2}} = W'(u),
\end{equation*}
which is nothing but  \eqref{corretta}, ``normalizing'' the constants $\hbar$  and $2m$  and taking 
$$ q:=\frac{e^4}{c^2 \kappa^2}, \quad 
  g(u):=W'(u)-\omega u.
$$

\section{Proof of Theorem \ref{Main}}\label{se:main}

\subsection{Preliminaries}

Following \cite{HIT}, without loss of generality we can replace condition \eqref{vecchiaitg1} with
\begin{enumerate}[label=(g\arabic*'),ref=g\arabic*',start=1]
\item \label{itg1'} $\displaystyle \limsup_{\xi\to\infty}\frac{g(\xi)}{e^{\alpha\xi^{2}}}=  0$, for all $\alpha>0$.
\end{enumerate}

Define
$$m_{0}:=-\frac{1}{2}\limsup_{\xi\to0}\frac{g(\xi)}{\xi}>0$$
and equip $\Hr$ with the norm $\|\cdot\|^2=\|\n \cdot\|^2_2+m_0\|\cdot\|_2^2$.\\

Consider $p_0\in (1,\infty)$ and set 
\[
\hspace{.89cm}{\lambda}(\xi)
=
\begin{cases}
\max\{g(\xi)+m_0\xi, 0\}, & \mbox{for } \xi\ge 0,\\
-\lambda(-\xi), & \mbox{for }  \xi<0,
\end{cases}
\qquad
{\Lambda}(\xi)=\int_{0}^\xi {\lambda}(\tau)d\tau, 
\]

\[
\bar{\lambda}(\xi)
=
\begin{cases}
\displaystyle 
\xi^{p_0}\sup_{0<\tau\le\xi} \frac{\lambda(\tau)}{\tau^{p_0}}, & \mbox{for }  \xi>0,\\
0, & \mbox{for }  \xi=0,\\
-\bar{\lambda}(-\xi), & \mbox{for }  \xi<0,
\end{cases}
\qquad
\bar{\Lambda}(\xi)=\int_{0}^\xi \bar{\lambda}(\tau)\,d\tau. 
\]
The functions $\l$, $\L$, $\bar{\lambda}$ and $\bar{\L}$ satisfy the following properties (see \cite[Lemma 2.1, Corollary 2.2, Lemma 2.3]{HIT}).

\begin{lemma}\label{proplambda}
The following holds
\begin{enumerate}[label=(\roman*),ref=\roman*]
\item $g(\xi)+m_0\xi\le \lambda(\xi)\le \bar{\lambda}(\xi)$, for all $\xi\ge 0$.
\item $\lambda(\xi)\ge 0$ and $\bar{\lambda}(\xi)\ge 0$, for all $\xi\ge 0$.
\item There exists $\delta_0>0$ such that $\lambda(\xi)=\bar{\lambda}(\xi)=0$ for $\xi\in [0,\delta_0]$.
\item There exists $\xi_0>0$ such that $0<\lambda(\xi_0)\le \bar{\lambda}(\xi_0)$.
\item The map $\xi\mapsto{\bar{\lambda}(\xi)}/{\xi^{p_0}}$: $(0,\infty)\rightarrow\mathbb R$ is non-decreasing. 
\item $\lambda(\xi)$, $\bar{\lambda}(\xi)$ satisfy \eqref{itg1'}.
\end{enumerate}
\end{lemma}

\begin{cor}\label{propLambda}
The following holds
\begin{enumerate}[label=(\roman*),ref=\roman*]
\item \label{i32}$G(\xi)+m_0|\xi|^2/2\le \L(\xi)\le \bar{\L}(\xi)$, for all $\xi\in\mathbb{R}$.
\item $\L(\xi),\bar{\L}(\xi)\ge 0$, for all $\xi\in\mathbb{R}$.
\item There exists $\delta_0>0$ such that $\L(\xi)=\bar{\Lambda}(\xi)=0$ for $\xi\le \delta_0$.
\item $\bar{\L}(\zeta_0)-m_0\zeta_0^2 /2>0$.
\item $0\le (p_0+1)\bar{\L}(\xi)\le\xi\bar{\l}(\xi)$, for all $\xi\in\mathbb{R}$.
\item $\Lambda(\xi)$, $\bar{\Lambda}(\xi)$ satisfy \eqref{itg1'}.
\end{enumerate}
\end{cor}

\begin{lemma}\label{lambdacomp}
Suppose that $ u_j\rightharpoonup  u_0$ in $H^1_r(\RD)$, then
\begin{enumerate}[label=(\roman*),ref=\roman*]
\item $\dis \ird\L(u_j)dx \to \ird \L(u_0) dx$, $\dis \ird\bar{\L}(u_j)dx \to \ird \bar{\L}(u_0) dx$;
\medskip
\item $\l(u_j) \to  \l(u_0)$, $\bar{\l}(u_j) \to  \bar{\l}(u_0)$ strongly in $H^{-1}_r(\RD)$.
\end{enumerate}
\end{lemma}

Now we recall some useful properties of the nonlocal term of the functional $J_q$ defined in \eqref{J}.
For any $u\in \Hr,$ let
\begin{equation*}
N(u) := \int_{\mathbb R^2} \frac{u^2(|x|)}{|x|^2}\Big(\int_0^{|x|}su^2(s)ds\Big)^2 dx.
\end{equation*}

We explicitly remark that
\begin{equation*} 
	N(u)
	\le
	C\int_{\mathbb R^2} u^2(|x|)\Big(\int_{B(0,|x|)}u^4(y)dy\Big) dx 
	\le C \|u\|^6 
\end{equation*}
and, for any $v\in \Hr$,
\[
N'(u)[v] 
=
2\int_{\mathbb{R}^2} \frac{u(|x|) v(|x|)}{|x|^{2}}\Big(\int_{0}^{|x|} su^{2}(s)ds\Big)^{2}dx
+4\int_{\mathbb{R}^2}\frac{u^{2}(|x|)}{|x|^{2}}\Big( \int_{0}^{|x|}su^{2}(s)ds\Big)\Big(\int_{0}^{|x|}su(s)v(s)ds\Big)dx. \\
\]
In particular,
\begin{equation}\label{+N'}
N'(u)[u] = 6\int_{\mathbb R^2} \frac{u^2(|x|)}{|x|^2}\Big(\frac{1}{2\pi}\int_{B(0,|x|)}u^2(y)dy\Big)^2 dx = 6N(u).
\end{equation}
Moreover, it is easy to see that the nonlocal term $N(u)$ has the following rescaling properties
$$ N(u(\tau \cdot))=\tau^{-4}N(u(\cdot))\quad \text{ and }\quad N'(u(\tau \cdot))[v(\tau \cdot)]=\tau^{-4}N'(u(\cdot))[v(\cdot)]\quad\hbox{for all } \tau>0.$$
We will make use of this property throughout the paper without any other comment.

By using the compact embedding of $H_r^1(\mathbb{R}^2)$ into $L^p(\mathbb{R}^2)$ for $p>2$, we get the following compactness properties of $N(u)$ 
(see \cite[Lemma 3.2]{BHS}).
\begin{lemma}\label{lem:compact}
	If $(w_{j})_j$ converges  weakly to $w$ in $H_r^{1}(\mathbb R^{2})$ as $j\rightarrow\infty$, then, up to a subsequence,
	the following convergences hold
	\begin{align*}
	& N(w_{j})\to N(w), \\
	& N'(w_{j})[w_{j}]\to N'(w)[w],\\
	& N'(w_{j})[v] \to N'(w)[v], \hbox{ for all } v\in H_r^{1}(\mathbb R^{2}).
	\end{align*}
\end{lemma}

\subsection{The minimax scheme}

Due to the presence of the nonlocal term and of a nonlinearity satisfying very general assumptions, there are several problems related to 
the geometry and the compactness of the functional $J_q$. So, we are not able to find solutions of \eqref{corretta} directly  and we  have
to modify the functional by means of a truncation.
Therefore, let  $\varphi\in C^{\infty}(\mathbb R_{+}, [0,1])$ satisfy
\begin{equation*} 
\begin{cases}
\varphi(s)=1, & \mbox{for } s\in [0,1],\\
\varphi(s)=0, & \mbox{for } s\in[2,\infty),\\
\|\varphi'\|_{\infty}\le 2
\end{cases}
\end{equation*}
and consider the perturbed functional $\mathcal J_q:H_{r}^{1}(\mathbb R^{2})\to \mathbb R$
\begin{equation*} 
\mathcal J_q(u)=\frac{1}{2}\|\nabla u\|_{2}^{2}+
\frac{q}{2}\varphi(q N(u))N(u)-
\int_{\mathbb R^{2}}G(u)\,dx
\end{equation*}
which is $C^1$ and
\begin{equation*} 
\mathcal J_q'(u)[v] =  \int_{\mathbb R^{2}} \nabla u \nabla v\,dx+
\frac{q^2}{2}\varphi'(q{N(u)})N(u) N'(u)[v] 
 +\frac{q}{2}\varphi(qN(u))N'(u)[v] -\int_{\mathbb R^{2}}g(u)v\,dx.
\end{equation*}

We are going to find critical points for $\mathcal J_q$. Evidently, if we have a bound of type $N(u)\le 1/q$ on the critical points 
of $\mathcal J_q$,  then they will be 
critical points of $J_q$ and hence solutions of the equation \eqref{corretta}. 
Moreover, let us define the following comparison functional
\begin{equation*} 
\mathcal I(u):=\frac{1}{2}\|\nabla u\|_{2}^{2}-
\int_{\mathbb R^{2}}\bar{\Lambda}(u)dx.
\end{equation*}
Observe that $\mathcal J_{q}\ge \mathcal I$ (by \eqref{i32} of Corollary \ref{propLambda}) and they are both even functionals.

As stated in the next lemma, $\mathcal J_{q}$ and $\mathcal I$ have the geometry of the Symmetric Mountain Pass Theorem. 
In what  follows, we set    $\mathbb D_{n}=\{\sigma\in \mathbb R^{n}: |\sigma|\le1\}$
and $\mathbb S^{n-1}=\partial \mathbb D_{n}.$

\begin{lemma}\label{JI}
For all $q>0$, the functionals $\mathcal J_{q}$ and $\mathcal I$ satisfy the following properties.
\begin{enumerate}[label=(\roman*),ref=\roman*]
\item \label{i35}There exist $r_{0}, \rho_{0}>0$ such that
\begin{align*}
\mathcal J_{q}(u) \ge \mathcal I(u)\ge 0,  & \text{ for }  \|u\|\le r_{0},\\ 
\mathcal J_{q}(u) \ge \mathcal I(u)\ge\rho_{0},   &\text{ for } \|u\|= r_{0}. 
\end{align*}
\item For every $n\in \mathbb N$ there exists an odd and continuous map $\gamma_{n}: \mathbb S^{n-1}\to H^{1}_{r}(\mathbb R^{2})$ 
such that
\begin{equation}\label{2b} 
\mathcal I(\gamma_{n}(\sigma))\le \mathcal J_{q}(\gamma_{n}(\sigma))<0.
\end{equation}
\end{enumerate}
\end{lemma}

\begin{proof}
The first part follows by \cite[Lemma 2.4]{HIT}.
To prove \eqref{2b}, we argue as in \cite[Theorem 10]{BL2}: for any $n\in \mathbb{N}$, 
an odd and continuous 
map $\pi_n: \mathbb S^{n-1}\rightarrow H_r^1(\mathbb{R}^2)$ is defined such that
\begin{equation*} 0\notin \pi_{n}(\mathbb S^{n-1})\ \mbox{and} \ \int_{\mathbb{R}^2} G(\pi_{n}(\sigma))\,dx\ge 1,\ \ \mbox{for all}
\ \sigma\in\mathbb S^{n-1}.
\end{equation*}
Then for $\theta$ sufficiently large, setting $\gamma_n(\sigma):= \pi_n(\sigma)(\cdot/\theta)$, we have
\begin{align*}
\mathcal J_q(\gamma_n(\sigma))
&=\frac{1}{2}\int_{\mathbb{R}^2}|\nabla \pi_n(\sigma)|^2\,dx+\frac{q}{2}\theta^4\varphi( q\theta^4 N(\pi_n(\sigma)))
N(\pi_n(\sigma)) -\theta^2\int_{\mathbb{R}^2}G(\pi_n(\sigma))\,dx\\
&\le \frac{1}{2}\int_{\mathbb{R}^2}|\nabla \pi_n(\sigma)|^2\,dx -\theta^2<0.
\end{align*}
\end{proof}

Due to Lemma \ref{JI}, for every $n\in \mathbb N$, we can define a family of mappings $\Gamma_n$ by
\begin{equation*} 
\Gamma_{n}:=\{\gamma\in C(\mathbb D_{n}, H^{1}_{r}(\mathbb R^{2})): \gamma(-\sigma)=-\gamma(\sigma) \ \
\text{and } \ \gamma_{|\mathbb S^{n-1}}=\gamma_{n}\},
\end{equation*}
which is nonempty since
\[
\alpha_{n}(\sigma)=
\begin{cases}
0, & \mbox{for }\sigma=0,\\
|\sigma| \gamma_{n}(\sigma/|\sigma|), & \mbox{for } \sigma\in \mathbb D_{n}\setminus \{0\},
\end{cases}
\]
belongs to $\Gamma_{n}$. 
\\
Now define the values
\begin{equation*}
b_{n}(q):=\inf_{\gamma\in \Gamma_{n}}\max_{\sigma\in \mathbb D_{n}} \mathcal J_{q}(\gamma(\sigma)).
\end{equation*}
A first property of these levels is the following estimate.
\begin{lemma}\label{AnB}
For all $q>0$ and $n\in \mathbb N$, there exists a constant $C(n)$, 
such that $b_{n}(q)\le C(n)$.
\end{lemma}

\begin{proof}
For a fixed $\gamma\in \Gamma_{n}$, we have 
\begin{align*}
b_{n}(q)
&\le
\max_{\sigma\in \mathbb D_{n}} \mathcal J_{q}(\gamma(\sigma)) \\
&\le
\max_{\sigma\in \mathbb D_{n}}\Big\{\frac{1}{2}\|\nabla \gamma(\sigma)\|_{2}^{2}-\int G(\gamma(\sigma))\Big\}
+\max_{\sigma\in \mathbb D_{n}}\Big\{\frac{q}{2}\varphi(qN(\gamma(\sigma)))N(\gamma(\sigma))\Big\}\\
&\le
\begin{cases}
C(n), & \textrm{if } N(\gamma(\sigma))\ge 2/q\\
C(n)+1,  & \textrm{if } N(\gamma(\sigma))< 2/q.
\end{cases}
\end{align*}
\end{proof}

\begin{lemma}\label{bninfty}
For all $q>0$, the values $b_{n}(q)$  are divergent 
and so we can assume that they are 
strictly monotone.
\end{lemma}
\begin{proof}
In \cite[Lemma 3.2]{HIT} it has been proved that  the values
$$c_{n}:=\inf_{\gamma\in \Gamma_{n}}\max_{\sigma\in \mathbb D_{n}} \mathcal I(\gamma(\sigma))\to+\infty\quad\hbox{as }n\to+\infty.$$
Moreover it is easy to see that 
for all $\gamma\in \Gamma_{n}$, $\gamma(\mathbb D_{n})\cap\left\{u\in H^{1}_{r}(\mathbb R^{2}): \|u\|=r_{0}\right\}\neq\emptyset$. 
The conclusion then follows since, by \eqref{i35} Lemma \ref{JI}, 
\begin{equation*} 
0<\rho_{0}\le c_{n}\le b_{n}(q).
\end{equation*}
\end{proof}

%

To deal with the lack of compactness, it is convenient to work in the augmented space $\mathbb R\times H^{1}_{r}(\mathbb R^{2}).$
For this, we define the extended functional
\begin{equation}\label{extended}
\begin{split}
\tilde{\mathcal J}_{q}(\theta,u):=\mathcal J_{q}(u(e^{-\theta}\cdot))=
\frac{1}{2}\|\nabla u\|_{2}^{2}+
\frac{q}{2}e^{4\theta}\varphi(qe^{4\theta}N(u))N(u)-e^{2\theta}\int_{\mathbb R^{2}} G(u)\,dx,
\end{split}
\end{equation} 
and its derivative will be denoted by
$\tilde {\mathcal J}_{q}'=
(\partial_{\theta} \tilde {\mathcal J}_{q}, \partial_{u}\tilde {\mathcal J}_{q})$
with
\begin{equation}\label{partialtheta}
\partial_{\theta} \tilde {\mathcal J}_{q}(\theta, u)=
2q e^{4\theta}\varphi(qe^{4\theta}N(u))N(u)
+2q^2 e^{8\theta} \varphi'(qe^{4\theta}N(u))N^2(u)
-2e^{2\theta}\int_{\mathbb R^2} G(u)\,dx,
\end{equation}
and
\begin{equation}
\label{partialu}
\begin{split}
\partial_{u} \tilde {\mathcal J}_{q}(\theta, u)[v] =&  \int_{\mathbb R^{2}} \nabla u \nabla v\,dx+
\frac{q^2}{2}e^{8\theta}\varphi'(qe^{4\theta}N(u))N(u) N'(u)[v] \\
& +\frac{q}{2}e^{4\theta}\varphi(qe^{4\theta}N(u))N'(u)[v] -e^{2\theta}\int_{\mathbb R^{2}}g(u)v\,dx.
\end{split}
\end{equation}
for all $v\in H_r^1(\mathbb R^{2})$.


Let us define the classes
\[
\tilde\Gamma_{n}=\left\{\tilde\gamma=(\theta, \eta)\in C(\mathbb D_{n}, \mathbb R\times H^{1}_{r}(\mathbb R^{2})):
\theta \text{ is even}, \eta \ \text{is odd, and }\tilde \gamma_{|\mathbb S^{n-1}}=(0,\gamma_{n})\right\},
\]
where $\gamma_{n}$ is given in Lemma \ref{JI}, and the levels
\begin{equation*}
\tilde b_{n}(q):= \inf_{\gamma\in \tilde \Gamma_{n}}\max_{\sigma\in \mathbb D_{n}} \tilde {\mathcal J}_{q}(\gamma(\sigma)).
\end{equation*}
Arguing as in \cite[Lemma 4.1]{HIT}, we have
\begin{lemma}\label{bn}
For all $q>0$, we have $\tilde b_{n}(q)=b_{n}(q)$; hence $\tilde b_{n}(q) \to +\infty$. 
\end{lemma}

To show that $b_{n}(q)$ are critical values for $\mathcal J_{q}$, we begin by showing that $\tilde b_{n}(q)$,
i.e. $b_{n}(q)$ by Lemma \ref{bn}, are ``almost critical values'' for $\tilde {\mathcal J}_{q}$ with a further important property. 
The proof is similar to \cite[Proposition 4.2]{HIT}.
\begin{proposition}\label{propPS}
For all $q>0$ and $n\in \mathbb N$, there exists a (PS) sequence $(\theta_{j}^{(n,q)}, u_{j}^{(n,q)})_{j}$
for $\tilde {\mathcal J}_{q}$ at level $ b_{n}(q)$ such that $\theta_{j}^{(n,q)}\to 0$ as $j\to\infty.$
\end{proposition}

Thus we are ready to prove the following fundamental result.
\begin{proposition}
\label{stime}
Let us fix $n\in \mathbb N$. There exists $\bar q(n)>0$ such that for every $q\in(0,\bar q(n))$, if $(\theta_{j}^{(n,q)}, u^{(n,q)}_{j})_{j}$ is a 
(PS) sequence as in Proposition \ref{propPS}, then $(u^{(n,q)}_{j})_{j}$ is bounded in $H^1_r(\mathbb{R}^2)$ uniformly with respect to $q$.
Furthermore, possibly passing to a subsequence, it converges to a critical point $u^{(n,q)}$ of  $\mathcal J_{q}$.
In particular $b_{n}(q)$ is  a critical value for $\mathcal J_{q}$.
\end{proposition}


\begin{proof}
Since $\tilde {\mathcal J}_{q}(\theta_{j}^{(n,q)},u^{(n,q)}_{j})=b_{n}(q) + o_j(1)$ and $\partial_{\theta}\tilde {\mathcal J}_{q}(\theta_{j}^{(n,q)}, 
u^{(n,q)}_{j})=o_j(1)$, then 
\[
2\tilde {\mathcal J}_{q}(\theta^{(n,q)}_{j},u^{(n,q)}_{j})
-\partial_{\theta} \tilde {\mathcal J}_{q}(\theta^{(n,q)}_{j}, u^{(n,q)}_{j})=2b_{n}(q)+o_j(1),
\]
which is equivalent, by using \eqref{extended} and \eqref{partialtheta}, to
\[
\|\nabla u_{j}^{(n,q)}\|_{2}^{2}=2b_{n}(q)+C_{j}^{(n,q)}+D_{j}^{(n,q)}+o_j(1)
\]
where
\begin{align*}
C_{j}^{(n,q)}
&:=
qe^{4\theta_{j}^{(n,q)}}\varphi(qe^{4\theta_{j}^{(n,q)}}N(u_{j}^{(n,q)}))N(u_{j}^{(n,q)})\\
D_{j}^{(n,q)}
&:=
2q^2e^{8\theta_{j}^{(n,q)}} \varphi'(qe^{4\theta_{j}^{(n,q)}}N(u^{(n,q)}_{j})) N^{2}(u_{j}^{(n,q)}).
\end{align*}
We easily see that
\begin{align*}
\text{if } qe^{4\theta^{(n,q)}_{j}} N(u_{j}^{(n,q)})\ge 2,
& \text{ then }   C_{j}^{(n,q)}=D_{j}^{(n,q)}=0, \\
\text{if } qe^{4\theta^{(n,q)}_{j}} N(u_{j}^{(n,q)})< 2, 
& \text{ then } C_{j}^{(n,q)}< 2,\ D_{j}^{(n,q)}< 16,
\end{align*}
so, in any case, by Lemma \ref{AnB}, 
\[ 
\|\nabla u^{(n,q)}_j\|_2^2\le C(n). 
\]
We show now that there exist $\bar q(n)>0$ and $C(n)>0$ such that for all $q\in (0,\bar q(n))$, there exists $j_0\in \N$ such that for all  $j\ge j_0$, $\|u^{(n,q)}_{j}\|^{2}_{2}\le C(n)$: this will prove the uniform boundedness of $(u_j^{(n,q)})_j$ in $\Hr$.
\\
Arguing by contradiction, let us assume that
\beq\label{assurdo}
\forall k\in\mathbb{N}^* \ \exists q_k\in (0,1/k)\  \hbox{s.t.}\ \forall h\in \N \ \exists j_{k,h}\ge h : \|u^{(n,q_k)}_{j_{k,h}}\|^{2}_{2}> k.
\eeq
Let $k=1$ and consider the associated $q_1$. By Proposition \ref{propPS}, there exists $j^1\in \N$ such that $|\t_j ^{(n,q_1)}|<1$ and 
$\|\partial_u\tilde{\mathcal J}_{q_1}(\t_j ^{(n,q_1)}, u_j ^{(n,q_1)})\|<1$, for all $j\ge j^1$.
Hence, by \eqref{assurdo}, taking $h_1=\max\{j^1,1\}$, there exists $j_1\ge h_1$ such that $\|u^{(n,q_1)}_{j_1}\|^{2}_{2}> 1$.
\\
Now let $k=2$ and consider the associated $q_2$. By Proposition \ref{propPS}, there exists $j^2>j_1$ such that $|\t_j ^{(n,q_2)}|<1/2$ and 
$\|\partial_u\tilde{\mathcal J}_{q_2}(\t_j ^{(n,q_2)}, u_j ^{(n,q_2)})\|<1/2$, for all $j\ge j^2$.
Again, by \eqref{assurdo}, taking $h_2=\max\{j^2,2\}$, there exists $j_2\ge h_2$ such that $\|u^{(n,q_2)}_{j_2}\|^{2}_{2}> 2$.
\\
By an iterative procedure, for all $k\in \N$, there exists $j_k\ge k$ such that 
\beq \label{stranaps}
|\t_{j_k} ^{(n,q_k)}|<\frac 1k,\quad
\|\partial_u\tilde{\mathcal J}_{q_k}(\t_{j_k} ^{(n,q_k)}, u_{j_k} ^{(n,q_k)})\|<\frac 1k, 
\quad
\|u^{(n,q_k)}_{j_k}\|^{2}_{2}> k.
\eeq
For the sake of brevity, we rename the sequence that satisfies \eqref{stranaps} by  $(\tilde{\t}_j,\tilde{u}_j)_j$.  
Since $\partial_u\tilde{\mathcal J}_{q_j}(\tilde{\theta}_j, 
\tilde{u}_j)\rightarrow 0$, then for all $v\in H_r^1(\mathbb R^2)$,
\begin{equation*}
 | \partial_u\tilde{\mathcal J}_{q_j}(\tilde{\t}_j, \tilde{u}_j)[ v]  |\le \varepsilon_j \|v\|,
\end{equation*}
where
\begin{equation*} \varepsilon_j:=\|\partial_u\tilde{\mathcal J}_{q_j}(\tilde{\t}_j, \tilde{u}_j)\|\rightarrow 0.
\end{equation*}
More precisely, using \eqref{partialu},
\begin{equation}\label{big1}
\begin{split}
& \Big| \int_{\mathbb R^{2}} \nabla \tilde{u}_j \nabla v \,dx+
\frac{q_j^2}{2}e^{8\tilde{\t}_j}\varphi'(q_j e^{4\tilde{\t}_j}N(\tilde{u}_j))N(\tilde{u}_j) N'(\tilde{u}_j)[v] \\
& \quad+\frac{q_j}{2}e^{4\tilde{\t}_j}\varphi(q_j e^{4\tilde{\t}_j}N(\tilde{u}_j))N'(\tilde{u}_j)[v] 
-e^{2\tilde{\t}_j}\int_{\mathbb R^{2}}g(\tilde{u}_j)v\,dx \Big|\\
&\le \varepsilon_j\sqrt{\|\nabla v\|_2^2+m_0\|v\|_2^2}.
\end{split}
\end{equation}
Consider $t_j=1/\|\tilde{u}_j\|_2\rightarrow 0$ as $j\rightarrow\infty$ and 
$\widehat{u}_j(\cdot)=\tilde{u}_j(\cdot / t_j)$. Thus we have
\begin{align}
 & \|\widehat{u}_j\|_2^2 =1,\label{hat-u1}\\
 & \|\nabla \widehat{u}_j\|_2^2 =\|\nabla\tilde {u}_j\|_2^2\le C(n). \label{hat-u2}
\end{align}
Then $(\widehat{u}_j)_j$ is bounded in $H_r^1(\mathbb R^2)$ and so, along
a subsequence, $\widehat{u}_j\rightharpoonup \widehat u_0$ weakly in $H_r^1(\mathbb R^2)$.\\
\textit{Claim.} $\widehat u_0=0$.\\
Let $v\in C_0^{\infty}(\mathbb R^2)$ and $\widehat v(\cdot):= v({t_j} \cdot )$. Evaluating \eqref{big1} with $\widehat v$ and multiplying by $t_j^2$, we obtain
\begin{equation}\label{big3}
\begin{split}
& \Big| t_j^2\int_{\mathbb R^{2}} \nabla \widehat{u}_j \nabla  v\,dx+
\frac{q_j^2}{2}\frac{e^{8\tilde{\t}_j}}{t_j^6}\varphi'\Big(\frac{q_j e^{4\tilde{\t}_j} N(\widehat{u}_j)}{t_j^{4} }\Big)
N(\widehat{u}_j) N'(\widehat{u}_j)[ v] \\
& \quad+\frac{q_j}{2}\frac{e^{4\tilde{\t}_j}}{t_j^2}\varphi\Big(\frac{q_j e^{4\tilde{\t}_j}N(\widehat{u}_j)}{t_j^4}\Big)
N'(\widehat{u}_j)[ v] 
-e^{2\tilde{\t}_j}\int_{\mathbb R^{2}}g(\widehat{u}_j) v\,dx \Big|\\
&\le \varepsilon_j t_j\sqrt{t_{j}^{2}\|\nabla v\|_2^2+m_0\| v\|_2^2}.
\end{split}
\end{equation}
Now we distinguish some cases as follows.\\
\textit{Case 1}. $q_j e^{4\tilde{\t}_j}N(\widehat{u}_j)\le 2t_j^{4}$.\\
\textit{Case 1.1}. 
If ${t_j^{4}}/{q_j}\to 0$ then $N(\widehat{u}_j)\rightarrow 0$ as $j\rightarrow \infty$. From \cite[Proposition 2.4]{BHS} and \eqref{hat-u2}, 
we have 
\[
\|\widehat{u}_j\|_{4}^4\le 2\|\nabla \widehat{u}_j\|_{2}N^{{1}/{2}}(\widehat{u}_j)\to 0
\]
and hence
we deduce
$\widehat u_0\equiv 0$.\\
\textit{Case 1.2}. 
If ${t_j^{4}}/{q_j}\ge C$,
by the compactness of $N$ and $N'$ (see Lemma \ref{lem:compact}) 
and  the definition of $\varphi$, we infer that 
\begin{align*}
&\left|\frac{q_j^2}{2}\frac{e^{8\tilde{\t}_j}}{t_j^6}\varphi'\Big(\frac{q_j e^{4\tilde{\t}_j} N(\widehat{u}_j)}{t_j^{4} }\Big)
N(\widehat{u}_j) N'(\widehat{u}_j)[ v]\right|
\le C  t_j^2,
\\
& \left|\frac{q_j}{2}\frac{e^{4\tilde{\t}_j}}{t_j^2}\varphi\Big(\frac{q_j e^{4\tilde{\t}_j}N(\widehat{u}_j)}{t_j^4}\Big)
N'(\widehat{u}_j)[ v] \right|
\le C  t_j^2.
\end{align*}
Hence, since  $t_j\rightarrow 0$ as $j\rightarrow\infty$, by \eqref{big3}, 
we have $\int_{\mathbb R^{2}}g(\widehat u_0) v\,dx=0$ for all $ v\in C^\infty_0(\mathbb R^2)$.
From this and by condition \eqref{vecchiaitg2'}, $\widehat u_0\equiv 0$.\\
\textit{Case 2}. $\displaystyle\frac{q_j e^{4\tilde{\t}_j} N(\widehat{u}_j)}{t_j^{4} }> 2$.\\
Equation \eqref{big3}, in this case, becomes
\[
\Big| t_j^2\int_{\mathbb R^{2}} \nabla \widehat{u}_j \nabla  v\,dx
-e^{2\tilde{\t}_j}\int_{\mathbb R^{2}}g(\widehat{u}_j) v\,dx \Big|
\le \varepsilon_jt_j\sqrt{t_j^2\|\nabla v\|_2^2+m_0\| v\|_2^2}
\]
and we conclude simply repeating the arguments in Case 1.2 completing so the proof of the claim. \\
Now evaluating \eqref{big1} in $\tilde u_j$ and multiplying by  $t_j^2$ we obtain
\begin{equation*} 
\begin{split}
&  t_j^2\|\nabla \widehat{u}_j\|_2^2+
\frac{q_j^2}{2}\frac{e^{8\tilde{\t}_j}}{t_j^6}\varphi'\Big(\frac{q_j e^{4\tilde{\t}_j} N(\widehat{u}_j)}{t_j^{4} }\Big)
N(\widehat{u}_j) N'(\widehat{u}_j)[\widehat{u}_j]\\
& \quad+\frac{q_j}{2}\frac{e^{4\tilde{\t}_j}}{t_j^2}\varphi\Big(\frac{q_j e^{4\tilde{\t}_j}N(\widehat{u}_j)}{t_j^4}\Big)
N'(\widehat{u}_j)[\widehat{u}_j] 
-e^{2\tilde{\t}_j}\int_{\mathbb R^{2}}g(\widehat{u}_j)\widehat{u}_j\,dx 
= o_j(1).
\end{split}
\end{equation*}
Thus, by Lemmas \ref{proplambda} and \ref{lambdacomp} and \eqref{+N'}
\begin{equation}\label{big5}
\begin{split}
& t_j^2 \|\nabla \widehat{u}_j\|_2^2+ m_0 e^{2\tilde{\t}_j}\|\widehat{u}_j\|_2^{2}+
3q_j^2\frac{e^{8\tilde{\t}_j}}{t_j^6}\varphi'\Big(\frac{q_j e^{4\tilde{\t}_j} N(\widehat{u}_j)}{t_j^{4} }\Big)
N^2(\widehat{u}_j)
+3q_j\frac{e^{4\tilde{\t}_j}}{t_j^2}\varphi\Big(\frac{q_j e^{4\tilde{\t}_j}N(\widehat{u}_j)}{t_j^4}\Big)
N(\widehat{u}_j)\\
& 
= e^{2\tilde{\t}_j}\int_{\mathbb R^{2}}\Big(g(\widehat{u}_j)\widehat{u}_j+m_0(\widehat{u}_j)^2\Big)\,dx+o_j(1)
\le e^{2\tilde{\t}_j}\int_{\mathbb{R}^2}\lambda(\widehat{u}_j)\widehat{u}_j\,dx + o_j(1)\rightarrow 0.
\end{split}
\end{equation}
Now we show that  
\begin{equation*}
3q_j^2\frac{e^{8\tilde{\t}_j}}{t_j^6}\varphi'\Big(\frac{q_j e^{4\tilde{\t}_j} 
N(\widehat{u}_j)}{t_j^{4} }\Big) N^2(\widehat{u}_j)\rightarrow 0.
\end{equation*} 
If $q_j e^{4\tilde{\t}_j} N(\widehat{u}_j)\ge 2t_j^{4}$, then the desired convergence follows easily. On the other hand, 
if we have  $q_je^{4\tilde{\t}_j} N(\widehat{u}_j) <  2t_j^{4}$, then 
\begin{equation*}
\Big|3q_j^2\frac{e^{8\tilde{\t}_j}}{t_j^6}\varphi'\Big(\frac{q_j e^{4\tilde{\t}_j} 
N(\widehat{u}_j)}{t_j^{4} }\Big) N^2(\widehat{u}_j)\Big|\le C  t_j^2\rightarrow 0
\end{equation*}
Therefore, from \eqref{big5}, since
\begin{equation*}3q_j\frac{e^{4\tilde{\t}_j}}{t_j^2}\varphi\Big(\frac{q_j e^{4\tilde{\t}_j}N(\widehat{u}_j)}{t_j^4}\Big)
N(\widehat{u}_j)\ge 0,
\end{equation*} 
we get $\|\widehat{u}_j\|_2\rightarrow 0$, which is in contradiction with \eqref{hat-u1}.
Hence $(u_j^{(n,q)})_j$ is bounded in $H_r^{1}(\mathbb R^2)$ uniformly with respect to $q$. So, we can assume  that there exists 
$u^{(n,q)}\in H^{1}_{r}(\mathbb R^{2})$ such that, for $j\to\infty$,
\begin{equation}\label{a1}
u_{j}^{(n,q)}\rightharpoonup u^{(n,q)}
\text{ in } H_{r}^{1}(\mathbb R^{2}), 
\quad
u_{j}^{(n,q)}\rightarrow u^{(n,q)}
\text{ in } L^{p}(\mathbb R^{2}),\ p>2,
\quad
u_{j}^{(n,q)}\rightarrow u^{(n,q)}
\text{ a.e. in } \mathbb R^{2}.
\end{equation}
By hypotheses we know that for any $v\in H_r^1(\mathbb R^2)$,
\begin{equation}\label{b1}
\begin{split}
\partial_u \tilde{\mathcal J}_q(\theta_j^{(n,q)},u_j^{(n,q)})[v]
&=
\int_{\mathbb R^{2}} \nabla {u}_j^{(n,q)} \nabla v\,dx
+\frac{q^2}{2}e^{8\theta_j^{(n,q)}}\varphi'\big(qe^{4\theta_j^{(n,q)}}N({u}_j^{(n,q)})\big)N({u}_j^{(n,q)}) N'({u}_j^{(n,q)})[v]\\
&\quad
+\frac{q}{2}e^{4\theta_j^{(n,q)}}\varphi\big( qe^{4\theta_j^{(n,q)}}N({u}_j^{(n,q)})\big)N'({u}_j^{(n,q)})[\phi] 
-e^{2\theta_j^{(n,q)}}\int_{\mathbb R^{2}}g({u}_j^{(n,q)})\phi\,dx\rightarrow 0
\end{split}
\end{equation}  
as $j\rightarrow \infty$. Then by \eqref{a1} and Lemma \ref{lem:compact}, $u^{(n,q)}$ satisfies
\begin{align*}
&  \int_{\mathbb R^{2}} \nabla {u}^{(n,q)} \nabla v\,dx+
\frac{q^2}{2}\varphi'\big(q N({u}^{(n,q)})\big)N({u}^{(n,q)}) N'({u}^{(n,q)})[v] \\
&\quad
+\frac{q}{2}\varphi\big( qN({u}^{(n,q)})\big)N'({u}^{(n,q)})[v] 
-\int_{\mathbb R^{2}}g({u}^{(n,q)})v\,dx = 0 
\end{align*}
for every $v\in H_r^1(\mathbb R^2)$. In particular
\begin{equation}\label{b2}
\begin{split}
& \|u^{(n,q)}\|^2 + \frac{q^2}{2}\varphi'\big(qN({u}^{(n,q)})\big)N({u}^{(n,q)}) N'({u}^{(n,q)})[u^{(n,q)}]\\
&\quad
+\frac{q}{2}\varphi\big( qN({u}^{(n,q)})\big)N'({u}^{(n,q)})[{u}^{(n,q)}] 
-\int_{\mathbb R^{2}}\Big( g({u}^{(n,q)}){u}^{(n,q)} +m_0 (u^{(n,q)})^2\Big)\,dx = 0.
\end{split}
\end{equation}
Now, by considering $v=u_j^{(n,q)}$ in \eqref{b1}, we have
\begin{equation}\label{b3}
\begin{split}
& \|\nabla {u}_j^{(n,q)}\|^2+m_0 e^{2\theta_j^{(n,q)}}\|u_j^{(n,q)}\|_2^2+
\frac{q^2}{2}e^{8\theta_j^{(n,q)}}\varphi'\big(q e^{4\theta_j^{(n,q)}}N({u}_j^{(n,q)})\big) N({u}_j^{(n,q)}) N'({u}_j^{(n,q)})[{u}_j^{(n,q)}] \\
&\quad
+\frac{q}{2}e^{4\theta_j^{(n,q)}}\varphi\big( qe^{4\theta_j^{(n,q)}}N({u}_j^{(n,q)})\big) N'({u}_j^{(n,q)})[{u}_j^{(n,q)}] 
\\
& = e^{2\theta_j^{(n,q)}}\int_{\mathbb R^{2}}\lambda(u_j^{(n,q)})u_j^{(n,q)}\,dx-  
e^{2\theta_j^{(n,q)}}\int_{\mathbb R^{2}}\Big(\lambda(u_j^{(n,q)})u_j^{(n,q)}
- g({u}_j^{(n,q)}){u}_j^{(n,q)} 
- m_0(u_j^{(n,q)})^2\Big) \,dx+ o_j(1).
\end{split}
\end{equation}
From Lemma \ref{proplambda}, for every
$j\in\mathbb N$ and $x\in\mathbb{R}^2$, $$\lambda(u_j^{(n,q)}(x))u_j^{(n,q)}(x)-g(u_j^{(n,q)}(x))u_j^{(n,q)}(x)-m_0(u_j^{(n,q)}(x))^2\ge 0.$$
Thus, by Fatou's Lemma,
\begin{equation}\label{b4}
\begin{split}
&\liminf_{j\rightarrow\infty}\int_{\mathbb R^{2}}\Big(\lambda(u_j^{(n,q)})u_j^{(n,q)}- m_0(u_j^{(n,q)})^2- g({u}_j^{(n,q)}){u}_j^{(n,q)} \Big) \,dx\\
&\quad \ge \int_{\mathbb R^{2}}\Big(\lambda(u^{(n,q)})u^{(n,q)}- m_0(u^{(n,q)})^2- g({u}^{(n,q)}){u}^{(n,q)} \Big) \,dx.
\end{split}
\end{equation}
Due to Lemma \ref{lambdacomp}, we have 
\begin{equation}\label{b5}
\int_{\mathbb R^{2}}\lambda(u_j^{(n,q)})u_j^{(n,q)}\,dx\rightarrow \int_{\mathbb R^{2}}\lambda(u^{(n,q)})u^{(n,q)}\,dx.
\end{equation}
Finally, using 
\eqref{b2}, \eqref{b3}, \eqref{b4}, \eqref{b5} and Lemma \ref{lem:compact}, we deduce that
\begin{align*}
\limsup_{j\rightarrow\infty}\|u_j^{(n,q)}\|^2
&=
\limsup_{j\rightarrow\infty}\Big[ 
-\frac{q^2}{2}e^{8\theta_j^{(n,q)}}\varphi'\big(q e^{4\theta_j^{(n,q)}}N({u}_j^{(n,q)})\big) N({u}_j^{(n,q)}) N'({u}_j^{(n,q)})[{u}_j^{(n,q)}]\\
&\quad
-\frac{q}{2}e^{4\theta_j^{(n,q)}}\varphi\big( e^{4\theta_j^{(n,q)}}N({u}_j^{(n,q)})\big) N'({u}_j^{(n,q)})[{u}_j^{(n,q)}] 
+ e^{2\theta_j^{(n,q)}}\int_{\mathbb R^{2}}\lambda(u_j^{(n,q)})u_j^{(n,q)}\,dx\\
&\quad
-   e^{2\theta_j^{(n,q)}}\int_{\mathbb R^{2}}\Big(\lambda(u_j^{(n,q)})u_j^{(n,q)}- g({u}_j^{(n,q)}){u}_j^{(n,q)}- m_0(u_j^{(n,q)})^2 \Big) \,dx
 \Big]\\
&\le  
-\frac{q^2}{2}\varphi'\big(qN({u}^{(n,q)})\big) N({u}^{(n,q)}) N'({u}^{(n,q)})[{u}^{(n,q)}]
-\frac{q}{2}\varphi\big( qN({u}^{(n,q)})\big) N'({u}^{(n,q)})[{u}^{(n,q)}] \\
&\quad
+\int_{\mathbb R^{2}}\Big( g({u}^{(n,q)}){u}^{(n,q)} + m_0(u^{(n,q)})^2 \Big) \,dx
\\
&=  \|u^{(n,q)}\|^{2}.
\end{align*}



\noindent So, $u_{j}^{(n,q)}\to u^{(n,q)}$ in $H^{1}(\mathbb R^{2})$. Hence $\tilde{\mathcal{J}}_q (0, u^{(n,q)})=b_n(q)$  and $\tilde{\mathcal{J}}'_q (0, u^{(n,q)})=0$, completing the proof.
\end{proof}

\subsection{Conclusion of the proof}
Now the proof of Theorem  \ref{Main} can be concluded.
Let $n\in\mathbb{N}^*$ and consider $b_{1}(q)<\ldots< b_{n}(q)$ 
and the corresponding critical points ${u}^{(n,q)}_1,\ldots,{u}^{(n,q)}_n$ of ${\mathcal{J}}_q$. In fact, we show that, for $q$ small enough, 
${u}^{(n,q)}_1,\ldots,{u}^{(n,q)}_n$ are critical points of $J_q$ and so solutions of \eqref{corretta}.  
Indeed, by Proposition \ref{stime} we know that there exists $C(n)>0$ such that, for every $i=1,\ldots,n$,  $\|{u}^{(n,q)}_i\|\le C(n)$.
If by contradiction, there exists $i=1, \ldots, n$ such that $q N({u}^{(n,q)}_i) > 1$, we have
\[
\frac{1}{q} < N({u}^{(n,q)}_i) \le C \|{u}^{(n,q)}_i\|^6 \le  C(n),
\]
but this is not possible for $q$ small.

%
%

\begin{remark}
We observe that, if we were looking for the existence of a solution, and not for a multiplicity result, we could follow the arguments 
of \cite{ADP1}, where a suitable combination of the monotonicty trick of \cite{JJ}, the penalization technique and a Poho\v{z}aev identity 
is performed. 
\end{remark}

\end{document}